\documentclass[11pt]{amsart}
\usepackage[utf8]{inputenc}
\usepackage{amssymb,url,xspace}
\usepackage[all]{xy}
\usepackage{hyperref}

\numberwithin{equation}{section}

\newtheorem*{theo}{Theorem}
\swapnumbers

\theoremstyle{plain}
\newtheorem{prp}[subsection]{Proposition}
\newtheorem{lem}[subsection]{Lemma}

\theoremstyle{definition}
\newtheorem{defn}[subsection]{Definition}

\newtheorem{rems}[subsection]{Remarks}

\newcommand{\ie}{\emph{i.e.}\@\xspace}
\newcommand{\df}{\colon}
\newcommand{\ra}{\rightarrow}
\newcommand{\xra}[1]{\xrightarrow{#1}}

\newdir{ >}{{}*!/-5pt/\dir{>}} 
\newcommand{\id}{1\kern -.35em 1}

\newcommand{\lmd}[1]{#1\text{-}\operatorname{mod}}

\newcommand{\CC}{\mathbb{C}}
\newcommand{\NN}{\mathbb{N}}
\newcommand{\QQ}{\mathbb{Q}}
\newcommand{\ZZ}{\mathbb{Z}}  

\newcommand{\cC}{\mathcal{C}}

\newcommand{\sfA}{\mathsf{A}}

\newcommand{\bfe}{\mathbf{e}}
\newcommand{\bff}{\mathbf{f}}
\newcommand{\bfg}{\mathbf{g}}
\newcommand{\bfm}{\mathbf{m}}

\newcommand{\ux}{\underline{x}}
\newcommand{\uy}{\underline{y}}
\newcommand{\uhy}{\underline{\hat{y}}}

\newcommand{\Lam}{\Lambda}
\newcommand{\Sig}{\Sigma}
\newcommand{\alp}{\alpha}
\newcommand{\bet}{\beta}
\newcommand{\gam}{\gamma}
\newcommand{\iot}{\iota}

\newcommand{\vph}{\varphi}

\newcommand{\dimv}{\underline{\dim}}

\newcommand{\add}{\operatorname{add}}
\newcommand{\End}{\operatorname{End}}
\newcommand{\Ext}{\operatorname{Ext}}
\newcommand{\Gr}{\operatorname{Gr}}
\newcommand{\Hom}{\operatorname{Hom}}
\newcommand{\ind}{\operatorname{ind}}
\newcommand{\Bi}{\operatorname{Im}}
\newcommand{\Ker}{\operatorname{Ker}}
\newcommand{\Obj}{\operatorname{Obj}}
\newcommand{\rad}{\operatorname{rad}}

\newcommand{\coker}{\operatorname{Coker}}
\newcommand{\coind}{\operatorname{coind}}

\newcommand{\ka}{k}

\title[Caldero-Chapoton Formula]%
{A Caldero-Chapoton Formula for generalized Cluster Categories}
\author{Salomón Domínguez}
\email{salomon@matem.unam.mx}
\author{Christof Geiss}
\email{christof@matem.unam.mx}
\urladdr{http://www.matem.unam.mx/christof}
\address{Instituto de Matemáticas, Universidad Nacional Autónoma de México,
Ciudad Universitaria, C.P. 04510, México D.F., MEXICO}
\date{November 6, 2013}

\begin{document}
\maketitle

\section{Introduction}
The aim of this note is to extend the range of a formula due to 
Caldero and Chapoton~\cite[Prp.~3.10]{CalCha06} from module categories 
over hereditary algebras of finite type to triangulated 2-Calabi-Yau categories
with a cluster tilting object. In view of the work of Palu~\cite{Palu08} 
it appeared natural to expect such a formula in this context. However,
the proof of the original statement does not carry over directly. We had
to reorganize the proof of the key statement in our Lemma~\ref{Lem:fibers}
and to review the relevant results from~\cite{Palu08}.

\begin{theo} Let $\cC$ be a triangulated 2-Calabi-Yau $\CC$-category
with suspension functor $\Sig$ and basic cluster tilting object 
$T=T_1\oplus\cdots\oplus T_n$. 
For an Auslander-Reiten triangle~\cite{Happel91}
$\Sig Z\ra Y\ra Z\ra \Sig^2 Z$ in $\cC$ we have
\[
C^T_{\Sig Z}\cdot C^T_Z= C^T_Y+1,
\]
where $C^T_?\df\Obj(\cC)\ra \ZZ[x_1^\pm,\ldots,x_n^\pm]$ denotes 
Palu's cluster character~\cite[Sec.~1]{Palu08}.
\end{theo}

Notice that in case $\dim\cC(Z,\Sig^2 Z)=1=\dim\cC(Z,Z)$ our result follows from
Palu's multiplication formula~\cite{Palu12}, otherwise our result looks rather
surprising in view of the complexity of this multiplication formula.
\medskip

Our formula reduces in many cases the complexity of explicit cluster character
computations: let $Z$ be an (indecomposable) object in an Auslander-Reiten 
component which is of type $\ZZ\sfA_\infty$ or a tube in a cluster category.
Based on the formula of Caldero and Chapoton, Dupont pointed out how to express
$X^T_Z$ in terms of certain Chebyshev polynomials which take as arguments
the values of $X^T_?$ on objects in the ``mouth'' of the corresponding 
component in case $\Lam:=\End_\cC(T)^{\text{op}}$, see for 
example~\cite[Thm.~5.1]{Dupont12}. Our result shows, that formulas of this kind 
hold without any restriction on $\cC$ or $T$.

\subsection{Preliminaries}
Let $\ka$ be an algebraically closed field, and $\cC$ be a triangulated 
2-Calabi-Yau $\ka$-category with suspension functor
$\Sig$. This  means that $\cC$ admits a Serre functor which we may
identify for our purpose with $\Sig^2$. Thus, at each indecomposable object
$Z\in\cC$ there ends and Auslander-Reiten triangle, and the Auslander-Reiten
translate $\tau$ may be identified with $\Sig$, see for 
example~\cite[I.2]{ReiVdB02}.

Suppose that $\cC$ has a basic cluster tilting object
$T=T_1\oplus\cdots\oplus T_n$ and set $\Lam:=\End_\cC(T)^{\text{op}}$.
We consider the functor $E:=\cC(T,\Sig -)\df\cC\ra\lmd{\Lam}$.
Recall that $E$ induces an equivalence of $\ka$-categories
$\cC/(\add(T))\xra{\sim}\lmd{\Lam}$, see~\cite[Prp.~2.1]{KelRei07}.
Moreover, each exact sequence in $\lmd{\Lam}$ can be ``lifted'' to 
an distinguished triangle in $\cC$ by~\cite[Lemma~3.1]{Palu08}.

For each $X\in\cC$ there exists an distinguished triangle 
\[
\oplus_{i=1}^n T_i^{m(i,X)}\ra \oplus_{i=1}^n T_i^{p(i,X)}\ra X\ra
\Sig \left(\oplus_{i=1}^n T_i^{m(i,X)}\right)
\]
and $\ind_T(X):=(p(i,X)-m(i,X))_{i=1,\ldots,n}\in\ZZ^n$ is well-defined,
see~\cite[Sec.~2.1]{Palu08}.

Following Palu~\cite{Palu08}, we have in  case $\ka=\CC$ a cluster character
\[
C^T_?\df \Obj(\cC)\ra \QQ[x_1^\pm,\ldots, x_n^\pm],\ 
X\mapsto \ux^{\ind_T(X)}\sum_{\bfe} \chi(\Gr_\bfe(EX))\ux^{B_T\cdot\bfe},
\]
which naturally extends the expression introduced by Caldero and Chapoton
in~\cite[Sec.~3.1]{CalCha06}. 
Here, $\Gr_\bfe(EX)$ denotes the quiver grassmanian of $\Lam$-submodules of
$EX$ with dimension vector $\bfe$ and $\chi$ is the topological Euler 
characteristic. Finally, $B_T\in\ZZ^{n\times n}$ is the matrix with entries
\[
(B_T)_{i,j}:= \dim \Ext^1_\Lam(S_i,S_j) - \dim \Ext^1_\Lam(S_j,S_i),
\]
where $S_i:=E\Sig^{-1}T_i/\rad E\Sig^{-1}T_i$ for $i=1,\ldots,n$ represents
the simple $\Lam$-modules. Observe, that this is indeed Palu's cluster
character pre-composed  with the suspension $\Sig$.

\subsection{Example}
Let $\Lam$ be a finite dimensional basic $\CC$-algebra as above. Set
\begin{multline*}
\bfg_M:=(\dim\Ext^1_\Lam(S_i,M)-\dim\Hom_\Lam(S_i,M))_{i=1,\ldots,n}\\
\text{and}\quad C'_M:=\ux^{\bfg_M}\cdot F_M(\uhy) \ \text{ for } M\in\lmd{\Lam}, 
\end{multline*}
see Definition~\ref{def:F-pol} and Section~\ref{ssec:concl-pf} below
for the missing definitions. 
For $\bfm=(m_1,\ldots,m_n)\in\NN_0^n$ set 
$T^\bfm:=T_1^{m_1}\oplus\cdots\oplus T_n^{m_n}$. If $Z=Z'\oplus T^\bfm$ with
$Z'$ having no direct summand in $\add(T)$ we see that 
$C^T_Z=C'_{EZ}\cdot \ux^\bfm$. 

Let $0\ra\tau M\ra N\ra M\ra 0$ be an Auslander-Reiten sequence in $\lmd{\Lam}$,
and $Z\in\cC$ indecomposable with $EZ=M$.
Then, for the Auslander-Reiten triangle $\Sig Z\ra Y\ra Z\ra\Sig^2Z$ in 
$\cC$ we have $EY=N$ and $E\Sig Z=\tau M$.
Thus, if $Y$ has a non-trivial direct summand from $\add(T)$, we do \emph{not}
have $X'_{\tau M} X'_M=X'_N+1$. An example of this behaviour can be found
for $T$ being a non-acyclic cluster tilting object in a cluster category of
type $\mathsf{A}_3$ and $M$ a simple $\Lam=\End_\cC(T)^{\text{op}}$-module.

\section{Proof of the Theorem}
The following key result and the strategy for its proof are essentially
due to Caldero and Chapoton~\cite[Lemma~3.11]{CalCha06}. We review 
their proof, since we need the result for arbitrary finite dimensional
$\ka$-algebras, rather than only for hereditary algebras of finite type.
The only digression from their argument is at the beginning of step~(3) 
of the proof below, since we cannot assume in general that
$\dim\Ext^1_\Lam(M,\tau M)=1$ for an indecomposable module $M$.

\begin{lem} \label{Lem:fibers}
Let $\Lam$ be a finite dimensional basic $\ka$-algebra   and
$0\ra L\xra{\iot} M\xra{\pi} N\ra 0$ an Auslander-Reiten sequence in 
$\lmd{\Lam}$.
Consider the morphism of projective varieties
\[
\xi_\bfg\df \Gr_\bfg(M) \ra \coprod_{\bfe+\bff=\bfg} \Gr_\bfe(L)\times\Gr_\bff(N),
U\mapsto (\iot^{-1}(U),\pi(U)).
\]
Then the fiber $\xi_\bfg^{-1}(A,C)$ is an affine space isomorphic to 
$\Hom_\Lam(C, L/A)$ if $(A,C)\neq (0,N)$, and $\xi_\bfg^{-1}(0,N)=\emptyset$.
\end{lem}

\begin{proof}
(1) Suppose first we had $B\in\xi^{-1}(0,N)$. Then $B\cong N$ and we would 
obtain a section to $\pi$, a contradiction. Thus $\xi_\bfg^{-1}(0,N)=0$.

(2) Now, let $C\subset N$ be a proper submodule and denote
by $\gam\df C\hookrightarrow N$ the inclusion. We obtain the following
commutative diagram with exact rows:
\[\xymatrix{
0\ar[r]& L\ar@{=}[d]\ar[r]^{\iot'}& M'\ar[d]_{\gam'}\ar[r]^{\pi'} & 
C\ar[d]^\gam\ar[r]&0\\
0\ar[r]& L\ar[r]_{\iot} & M\ar[r]_\pi &N\ar[r] &0
}\]
The first row splits since $\gam$ factors over $\pi$.
Since $\gam'$ is injective by the snake lemma, it induces an isomorphism
between
\[
\Gr_{A,C}(M'):= \{ B\in\Gr_\bfg(M')\mid {\iot'}^{-1}(B)=A\text{ and } \pi'(B)=C\}
\]
and $\xi_\bfg^{-1}(A,C)$. Thus, in this case we are reduced to the corresponding
statement in the easy case of a split exact sequence, see for 
example~\cite[Lemma~3.8]{CalCha06}. Note, that the proof of this Lemma
does not depend on any particular property of $\Lam$.

(3) Finally, let $C=N$ and $A\neq 0$. Denote by 
$\alp\df A\hookrightarrow L$ the inclusion, and by 
$\Phi:=[0\ra L\ra M\ra N\ra 0]\in\Ext^1_\Lam(N,L)\setminus\{0\}$ the
class of our Auslander-Reiten sequence. We claim that there exists 
$\Psi\in\Ext^1_\Lam(N,A)$ with $\Ext^1_\Lam(N,\alp)(\Psi)=\Phi$, \ie
we have a commutative diagram
\[\xymatrix{
0 \ar[r] &A \ar[d]_\alp\ar[r] & M'\ar[d]_{\alp'}\ar[r]^\nu 
&N\ar@{=}[d]\ar[r] & 0\\
0 \ar[r] &L \ar[r]_\iot& M\ar[r]_\pi & N\ar[r] & 0
}\]
with exact (non-split) rows. The class of the top row is then $\Psi\neq 0$.
Indeed, consider the short exact sequence 
$0\ra A\xra{\alp} L\xra{\bet} L/A\ra 0$;
since $\Phi$ is an almost split sequence  and the projection $\bet\df L\ra L/A$
is not a section, we have
\[
\Phi\in\Ker(\Ext^1_\Lam(N,\bet))=\Bi(\Ext^1_\Lam(N,\alp)).
\]
Note that $\alp'$ is injective by the snake lemma. In particular, 
$\Bi(\alp')\in\xi^{-1}_\bfg(A,N)$ for $\bfg=\dimv A+\dimv N$. Now we can argue 
as~\cite{CalCha06} in the end of the proof of  Lemma~3.11, where the authors 
show that
\begin{multline*} 
\Hom_A(N,L/A)\ra \xi^{-1}_\bfg(A,N), \\
\vph\mapsto\{\iot(l)+\alp'(m)\mid (l,m)\in L\times M', \bet(l)=\vph\nu(m)\}
\end{multline*}
is the required bijection.  Note, that this elementary argument does not
require $\Lam$ to be hereditary or of finite representation type.
\end{proof}

\begin{defn} \label{def:F-pol}
Let $\Lam$ be a finite dimensional basic $\CC$-algebra.
For a $\Lam$-module $M$ we define the $F$-\emph{polynomial} to be
the generating function for the Euler characteristic of all possible
quiver grassmanians, \ie
\[
F_M:=\sum_\bfe \chi(\Gr_\bfe(M)) \uy^\bfe\in\ZZ[y_1,\ldots,y_n]
\]
where the sum runs over all possible dimension vectors of submodules of
$M$. Moreover, we assume that $S_1,\ldots, S_n$ is a complete system of
representatives of the simple $\Lam$-modules, and we identify the classes
$[S_i]\in K_0(\Lam)$ with the natural basis of $\ZZ^n$. 
\end{defn}

\begin{prp} \label{prp:AR-F}
Let $\Lam$ be a finite dimensional basic $\CC$-algebra.
Then the following holds:
\begin{itemize}
\item[(a)] 
If $0\ra L\xra{\iot} M\xra{\pi} N\ra 0$ is an Auslander-Reiten sequence in 
$\lmd{\Lam}$, then
\[
F_L\cdot F_N= F_M+ \uy^{\dimv N}.
\]
\item[(b)]
For the indecomposable projective $\Lam$-module $P_i$ with
top $S_i$ we have
\[
F_{P_i} = F_{\rad P_i} + \uy^{\dimv P_i}
\]
for  $i=1,\ldots, n$.
\item[(c)]
For the indecomposable injective $\Lam$-module $I_j$ module with socle $S_j$
we have
\[
F_{I_j}= y_j\cdot F_{I_j/S_j} +1
\]
for $j=1,2,\ldots, n$.
\end{itemize}
\end{prp}

\begin{proof}
(a) In view of the well-known properties of Euler characteristics, see for
example~\cite[7.4]{GeLeSc07}, this is a direct consequence of 
Lemma~\ref{Lem:fibers} and the definition of F-polynomials.

Statements (b) and (c) are easy.
\end{proof}

\begin{rems} \label{rem:AR-TS}
Let $\cC$ be a triangulated $2$-Calabi-Yau category with cluster tilting
object $T$ and $\Lam:=\End_\Lam(T)^{\text{op}}$ as in the introduction.
For an  Auslander-Reiten triangle 
$\Sig Z\xra{\alp} Y\xra{\bet} Z\xra{\gam}\Sig^2 Z$  
in $\cC$, clearly $Z$ is  indecomposable. If we apply the functor
$E:=\cC(T,\Sig-)$ the following three cases can occur:
\begin{itemize}
\item[(a)] 
If $Z\not\in\add(T\oplus\Sig^{-1}T)$ the sequence
\[
0\ra E\Sig Z\xra{E\alp} E Y\xra{E\bet} E Z\ra 0
\]
is an Auslander-Reiten sequence in $\lmd{\Lam}$.
\item[(b)]
If $Z\cong \Sig^{-1} T_i$ for some $i$, then 
\[
E Y\cong \rad P_i,\quad E Z\cong P_i \quad (\text{and } E\Sig Z=0).
\]
\item[(c)] 
If $Z\cong T_j$ for some $j$, then
\[
E \Sig Z\cong I_j,\quad E Y\cong I_j/S_j \quad (\text{and } EZ=0).
\]
\end{itemize}
In particular, always one of the situations of Proposition~\ref{prp:AR-F}
applies. We leave the easy proof as an exercise.
\end{rems}

\begin{prp} 
Let $\cC$ be a triangulated $2$-Calabi-Yau $\ka$-category with a basic cluster
tilting object $T=T_1\oplus\cdots\oplus T_n$.
\begin{itemize}
\item[(a)]
For each object $Z\in\cC$ we have
\begin{equation}
-B_T\cdot \dimv_\Lam(E Z) = \ind_T(\Sig Z) +\ind_T(Z) \label{eq:ind1}.
\end{equation}
\item[(b)]
If $X\xra{\alp} Y\xra{\bet} Z\xra{\gam}\Sig X$ is a distinguished triangle in
$\cC$, then
\begin{equation}
\ind_T(Y)=\ind_T(X)+\ind_T(Z)+ B_T\cdot \dimv_\Lam(\Ker(E\alp)).
\end{equation}
\item[(c)]
If
$\Sig Z\xra{\alp} Y\xra{\bet} Z\xra{\gam} \Sig^2 Z$ is an 
Auslander-Reiten triangle then moreover
\begin{equation}
\ind_T(Y) =
\begin{cases} \ind_T(\Sig Z) +\ind_T(Z) &\text{if } Z\not\in\add(T),\\
B_T\cdot\dimv_\Lam S_i &\text{if } Z\cong T_i
\end{cases} \label{eq:ind2}
\end{equation}
holds.
\end{itemize}
\end{prp}

\begin{proof}
(a)  We note first that we have
$\ind_T(\Sig Z)+\ind_T(Z)=\ind_T(\Sig Z)-\coind_T(\Sig Z)$
by~\cite[Lemma~2.1~(1)]{Palu08}.
Next, by~\cite[Lemma~2.3]{Palu08} and~\cite[Theorem~3.3]{Palu08}
together with our definition of $B_T$ we obtain
\[
\ind_T(\Sig Z)-\coind_T(\Sig Z)= B_T\cdot\dimv_\Lam EZ.
\]
Note that~\cite[Lemma~2.3]{Palu08} is stated for indecomposable objects
in $\cC$. However, for each indecomposable summand of $\Sig Z$ we obtain
the correct difference.

(b)  Choose $C\in\cC$ with $\cC(T,C)\cong\coker\cC(T,\bet)\cong \Ker(E\alp)$. 
By comparing with~\cite[Prop.~2.2]{Palu08} we have only to show 
that $\ind_T(C)+\ind_T(\Sig^{-1}C)= -B_T\cdot (\dimv_\Lam\cC(T,C))$. This
is exactly the statement of~(a) with $Z=\Sig^{-1} C$.

(c) Since we have an Auslander-Reiten triangle, a similar argument
as in Remark~\ref{rem:AR-TS} shows that $E\alp$ is injective
unless $Z\cong T_j$ for some $j=1,\ldots, n$. Thus, our claim follows
from~(b) if we note that $-\ind_T(\Sig T_j)=\ind_T(T_j)=\bfe_j$.  
\end{proof}

\subsection{Conclusion of the proof} \label{ssec:concl-pf}
Let $\hat{y}_j:=\prod_{i=1}^n x_i^{B_{i,j}}\in\ZZ[x_1^\pm,\ldots,x_n^\pm]$, 
and for $\bfe=(e_1,\ldots,e_n) \in\ZZ^n$ write 
$\uhy^\bfe:=\prod_{i=1}^n \hat{y}_i^{e_i}=\ux^{B_T\cdot\bfe}$. 
With this notation we have obviously
\begin{equation} \label{eq:C2}
C^T_X=\ux^{\ind_T(X)}F_{EX}(\uhy) \text{ for any } X\in\cC.
\end{equation}

Let first $Z\not\in\add(T)$, then we have
\begin{align*}
C^T_{\Sig Z}\cdot C^T_Z &= 
\ux^{\ind_T(\Sig Z)+\ind_T(Z)} F_{E\Sig Z}(\uhy)\cdot F_{EZ}(\uhy)\\
&=\ux^{\ind_T(\Sig Z)+\ind_T(Z)} (F_{EY}(\uhy)+\uhy^{\dimv_\Lam(E Z)})\\
&=\ux^{\ind_T(Y)}F_{EY}(\uhy) + 
\ux^{\ind_T(\Sig Z)+\ind_T(Z)+B_T\cdot \dimv_\Lam(EZ)}\\
&= C^T_Y+1.
\end{align*}
Here we used consecutively~\eqref{eq:C2}, Proposition~\ref{prp:AR-F}~(a) 
resp.~(b) and Remark~\ref{rem:AR-TS}~(a) resp.~(b), 
equation~\eqref{eq:ind2}, and finally~\eqref{eq:ind1}.

Finally, suppose that $Z\cong T_j$ for some $j$, \ie we consider the
Auslander-Reiten  triangle $\Sig T_j\ra Y \ra T_j\ra \Sig^2T_j$.
Then we have
\begin{align*}
C^T_{\Sig T_j}\cdot C^T_{T_j} &= x_j^{-1} F_{I_j}(\uhy)\, x_j\\
&= F_{I_j/S_j}(\uhy)\,\hat{y}_j+1\\
&= C^T_Y+1.
\end{align*}
Here, we used consecutively Remark~\ref{rem:AR-TS}, 
Proposition~\ref{prp:AR-F}~(c), and equation~\eqref{eq:ind2}.\qed

\subsection*{Acknowledgment} Both authors acknowledge partial support
from the CONACYT grant 81498.

\end{document}